\documentclass[10pt]{article}
\usepackage{stix2} 
 \usepackage[cal=boondoxo]{mathalfa} 
   \usepackage{amsmath,amsthm,amscd,graphicx}
 \usepackage{hyperref} 
\newtheorem{thm}{Theorem}[section]
\newtheorem*{thm*}{Theorem}
\newtheorem*{corr*}{Corollary}
\newtheorem{lemma}[thm]{Lemma}
\newtheorem{prop}[thm]{Proposition}
\newtheorem*{prop*}{Proposition}
\newtheorem{corr}[thm]{Corollary}
\theoremstyle{definition}

\newtheorem*{dfn*}{Definition}
\newtheorem{exmple}[thm]{Example}
\newtheorem*{exmple*}{Example}

\newtheorem*{conj*}{Conjecture}
\theoremstyle{remark}

\newtheorem{rmk}[thm]{\textit{Remark}}
\newcommand{\bbF}{{\mathbb{F}}}

\newcommand{\bP}{{\mathbb{P}}}
\newcommand{\bQ}{{\mathbb{Q}}}

\newcommand{\bZ}{{\mathbb{Z}}}
 \newcommand\cE{{\mathcal E}} 
 \newcommand{\comp}{\raise1pt\hbox{{$\scriptscriptstyle\circ$}}}
 \def\lset{\{}  
\def\rset{\}}  
\def\set#1{\lset#1\rset} 

\def\mapright#1{\mathop{\vbox{\ialign{
                                ##\crcr
    ${\scriptstyle\hfil\;\;#1\;\;\hfil}$\crcr
 \noalign{\kern2pt\nointerlineskip}
    \rightarrowfill\crcr}}\;}}

\def\half{\frac 12}

\def\ns#1{\operatorname{\mathsf {NS}}({#1})}

\begin{document}

\title{On the maximal number of du Val singularities for   a K3 surface}
\author{Chris Peters}
\maketitle

\begin{abstract}
A complex K3 surface or an algebraic K3 surface in characteristics distinct from $2$ cannot have more than $16$ disjoint nodal curves.
\end{abstract}

\section{Introduction}
It is well known that quartic surfaces can have at most $16$ ordinary double points   and  if this is the case the surface is a (classical) Kummer surface.
Indeed, projecting from a double point exhibits the quartic surface as a double cover of the plane branched in a degree $6$ curve with at most $15$ double points
and if there are $15$ double points,   the branch locus  consists of a union of $6$ lines in general position
and  its  double cover is a classical Kummer surface as explained in \cite[p. 774]{PAG}.   For a  proof   based on  affine geometry 
over the field $\bbF_2$, see \cite{nik2}.  See  also \cite[Prop. VIII.6.1]{4authors} where it is assumed that the $16$ double points form an "even" set (to be explained below). 

The aim of this note is to give a simple proof that any  complex K3 surface, algebraic or not,   can  acquire  no more than $16$ ordinary double points, a result originally due to V.\ Nikulin ~\cite[Applications, Cor. 1]{nik2} who proved this by lattice theoretic methods.
Remark that on the minimal resolution of such a surface the nodes give
  disjoint nodal curves and on a K3 surface  these  span a negative definite lattice in the Picard lattice, 
a priori  of rank $\le 20$ (recall that  non-algebraic surfaces are allowed) and so it  is somewhat surprising that there cannot be more than $16$ of them.

The proof is based on coding theory and  applies equally well to double points of  $A$-$D$-$E$  type, i.e. to   du Val  singularities where a refinement can be given. See Remark~\ref{rmk:refining}.
The idea of using coding theory to obtain bounds on the number of double points goes back to A. Beauville's article \cite{beau} 
and has   been further exploited in A.\ Kalker's thesis~\cite{kalker}. Applications of codes to
characteristic $2$ phenomena of K3 surfaces can be found in  I.\ Shimada's article~\cite{shima2}. 

The proof applies to algebraic as well as
non-algebraic complex K3 surfaces. The algebraic
nature of the proof ensures  that it is also valid for
  algebraic K3 surfaces in   characteristics distinct from $2$. See Remark~\ref{rmk:charp}.
  In characteristic $2$ the result  is false: there are supersingular K3 surfaces with $21$ nodes.
See \cite{shima}. 

\paragraph{Convention}
In this note $X$ is a compact complex surface, and in \S~\ref{sec:k3} it is a K3 surface. 
I assume  throughout that $H^2(X,\bZ)$ has no  torsion and I consider it as a lattice equipped with the integral pairing induced by the cup product.
In that situation the N\'eron--Severi group $\ns X$ is to be considered as a sublattice spanned by the classes of the divisors. This is a primitive sublattice of $H^2(X,\bZ)$. 

If $L$ is an integral lattice and $r\in \bQ$, I denote by $L(r)$ the same $\bZ$-module as $L$ but with form multiplied by $r$.

\paragraph{Acknowledgments} This note  materialized thanks to   ongoing  collaboration with Hans Sterk  for which I express my thanks. Email exchange 
with Daniel Huybrechts and Matthias Sch\"utt is also gratefully acknowledged.   Thanks go also to Viacheslav  Nikulin  who kindly  pointed me to his original proof.

\section{Codes and lattices}

A binary code is a linear subspace of some finite dimensional space $W$ over the field $\bbF_2$ which I shall identify with $\bbF_2^n$, $n=\dim W$.
A  vector $x$ of a  code $C$ is also called a  \emph{word}.  The number of non-zero coordinates of  $x$ is its \emph{weight}, $w(x)$.
The  dot-product  $x\cdot y= \sum x_j y_j \in \bbF_2$  
of two vectors  $ x=(x_1,\dots, x_n)$ and $y =(y_1,\dots,y_n)$ in $\bbF_q^n$  
defines a non-degenerate symmetric bilinear form and a code is called    isotropic 
 if $C\subset C^\perp$ and   self-dual  if $C=C^\perp$.  Via the  reduction modulo $2$ map
$
\rho:  \bZ^n  \to  \bbF_2^n $ a code $C$ lifts  to a   submodule $ \rho^{-1}C$ of $\bZ^n$  with
$
 2 \cdot \bZ^n\subset \rho^{-1}C \subset \bZ^n
$. 
The first inclusion shows that $\rho^{-1}C$ is of finite index in $\bZ^n$.    It  inherits the structure of a lattice from  $\bZ^n$ 
equipped with its dot-product. However,  it turns  out to be more convenient to use    a different lattice structure, namely   $\bZ^n$ equipped with
the standard euclidean form scaled by $\half$, which is   $\bZ^n \left(\half\right)$ by the convention I  adopted. This leads to the lattice
\begin{align}\label{eqn:LatFromCode}
\Gamma_C:= \rho^{-1} C \subset \bZ^n\left(\half\right).
\end{align}
One easily sees that the new product on $\Gamma_C$ has  integral values  if and only if $C$ is isotropic. See e.\ g.\ \cite[\S~1.3]{ebeling}.

%
%

I shall make use of the so-called Reed--Muller codes whose
  definition runs as follows.
Let $I$ be  a   finite set of size $N$. Then 
the functions $I \to \bbF_2$ form the $\bbF_2$-vector  space $\bbF_2^I$.
Suppose that $I$ itself consists of the points of an  $\bbF_2$-vector space $W $ of dimension $m$ so 
that $N= 2^m$.   I order these points as follows. Let $\set{e_0,\dots,e_{m-1}}$ be a basis for $W$ and
identify a point $x=\sum_{j=0}^{m-1} x_j e_j$ with the binary expansion  $n_x=\sum_{j=0}^{m-1} x_j 2^j$ of an integer between $0$ and $2^m-1$.  The natural order gives an ordering of the points of $W$.
A   function $f$ on $W$ determines a vector $(f_0,\dots, f_{N-1})\in \bbF_2^N$ as follows. First note that a point $x\in W$ determines the  unique
integer $j= n_x\in [0,\dots, N-1]$ and then I set  $f_j= f(x)$. 
The polynomial functions of degree $k$ on $W$ together with the zero function   define   a
subspace of $\bbF_2^N$ and this is also the case for   polynomial  functions of  degree  $\le k$. The latter define    the $k$-th order 
\textbf{\emph{Reed--Muller code}}  $S^{\le k} (W)\subset \bbF_2^N$, $N=2^m$. For an extensive treatment of these codes I refer to \cite[Chap. 4.5.]{vlint}.
 
\begin{exmple} Take $m=4$ and $k=1$. Then $N=2^4=16$ and $S^{\le 1}(\bbF_2^4)\subset \bbF_2^{16}$ is generated  by the   $4$ code words 
given as the rows  of the following
 $4\times 16$ matrix together with the vector with all coordinates equal to  $1$ arising from the constant function $1$. The columns correspond to the binary expansions of the numbers $0,...,15$
 and the rows correspond to the coordinate functions 
 $ x_0,x_1,x_2,x_3$:
\setcounter{MaxMatrixCols}{20}
\begin{align}
\label{eqn:RMcode}
\begin{pmatrix}
 0 &1 &0 & 1&0& 1 & 0&1&0&1&0&1&0&1&0&1 \\
 0&0&1&1& 0&0& 1&1& 0&0 &1&1& 0&0 &1&1 \\
 0&0&0&0&1&1&1&1&0&0&0&0&1&1&1&1 \\
 0&0&0&0&0&0&0&0&1&1&1&1&1&1&1&1
\end{pmatrix} .
\end{align}
\end{exmple}

The code \label{page:Dcode}
\begin{align} \label{eqn:CodeD}
\mathsf D_{m+1}:= S^{\le 1}  (W)
\end{align}
is a  code of dimension $m+1$ given by the affine linear functions on $W$.  One can view it 
as generated by the code   $\mathsf{C}_m  :=   S^1W $, consisting of  the linear functions on $W$,  together with  the   constant function $1$.
This last word has weight  $2^m$   while    the non-zero weights of $\mathsf{C}_m$ are  all  $2^{m-1}$ since the characteristic function of a hyperplane interpreted
as a code word has this weight. So the weights
of $\mathsf{D}_m$ itself (lowering the index by one) are  $0,2^{m-2}$ and $2^{m-1}$.
 %
The   code  $\mathsf{D}_m$  can be characterized as follows: 

\begin{lemma}[\protect{\cite[\S 4]{beau}}] 
  Let  $C\subset \bbF_2^n$ be  a code of dimension $m$ and with non-zero  weights   $\ge \half n$.
Then $n\ge 2^{m-1}$ and equality holds if and only if $C\simeq \mathsf{D}_m$ in which case
 it only has non-zero  weights $\half n= 2^{m-2}$ and $n=2^{m-1}$.
 \label{lemm:beaucodes}
\end{lemma}

I need a further property of the code  $\mathsf{D}_m$:
\begin{lemma} Suppose $  n>2^{m-1} $. There is no code in $\bbF^n_2$ with the property that   all coordinate subspaces of  $\bbF^n_2$
of dimension $2^{m-1}$ give a code isomorphic to  $\mathsf{D}_m$. \label{lemm:main}
\end{lemma}
\begin{proof} Set $N=2^{m-1}$.  Visualize the code  $\mathsf{D}_m\subset \bbF_2^N$ as generated by $(1,\dots,1)$  together with 
the rows of the $(m-1)\times N$ matrix, say $M$,  whose
columns are the binary expansions of the numbers $0,\dots, m-1$ ordered as in \eqref{eqn:RMcode}  where  $m=5, N=16$.

Suppose that I have a code   with the stated properties. Then it contains at least $(m-1)$ vectors whose first $N$ coordinates
form the matrix $M$. I now consider the corresponding vectors in $ \bbF^n_2$ forming a matrix $\widetilde M$. It suffices to consider the case $n=N+1$
since the effect of deleting $n-N-1$ columns from $\widetilde M$ and a further column from $M$ is the same for all $n\ge N+1$. 
Suppose  that the  last column of $\widetilde M$ corresponds to the decimal expansion  of $k\in [0,\dots, N-1]$.
Then this column is   the same as the $k$-th column $M_k$ of $M$.
Form now the  $(m-1)\times N$ matrix $M'$ by deleting from $\widetilde M$ a column $M_\ell$ of  the submatrix $M$ with $\ell\not=k$.
  Since all columns of $M$ are different, there is a row  $j$ such that
  the  entry  $M_{j \ell}$ of $M_\ell$ is distinct from the entry $M_{j k}$ of $M_k$. This implies  that the $j$-th row of the new matrix $M'$  has weight $\half N \pm 1$ which is a contradiction.
Hence there is no such code.
\end{proof}

\section{Sets of nodal curves on a surface}
Let $X$ be a compact complex surface containing a finite set $\cE=\set{E_1,\dots, E_n}$ of disjoint smooth rational curves
with $E_i^2=-2$, $i=1,\dots ,n$. Such curves are also called nodal curves, since these arise as minimal resolutions of ordinary double points.
Let  $L = \bZ   E_1 \operp \cdots\operp \bZ  E_n$ be   the abstract lattice with basis the nodal classes\footnote{Here $\operp$ denotes orthogonal direct sum.}
which can be identified with $\bZ^n(-2)$. Its dual is given by
$
   L^*= \half L \simeq \bZ^n\left(-\half\right)$.    The quotient map $L^* \to L^*/L =\half L/L \simeq \bbF_2^n$ is just the modulo $2$ map $\rho:\bZ^n\to \bbF_2^n$ used to construct lattices from codes 
   (see \eqref{eqn:LatFromCode}).
 Reversing  the procedure, one starts with the    lattice 
     \[
   N_\cE: =\text{ primitive closure of } L \text{ in } \ns X, 
   \]
   where $\ns X$ is the N\'eron--Severi lattice of $X$.
    Since there are inclusions
   $
   L  \subset N_\cE\subset N_\cE^*\subset L^*\simeq \bZ^n(-\half)$, 
    setting 
   \begin{equation}
   \label{eqn:NodeCode2}
   C_\cE= N_\cE/L \subset L^*/L = \bbF_2^n,
      \end{equation}
   the lattice $N_\cE(-1)$ is precisely  the inverse image of the code $ {C_\cE}$    under the mod $2$ map. In other words, 
     \begin{equation}
   \label{eqn:NodeCode3}
   N_\cE = \Gamma_{C_\cE}(-1).
   \end{equation}
    Using that $N_\cE$ is primitive in the lattice $\ns X $, I arrive at
    an  equivalent description of the code $C_\cE$, namely
     \begin{align*}
      \label{eqn:NodeCode4}
       C_\cE \simeq \ker \left(   \bbF_2^n= \oplus_j \bbF_2E_j \simeq   L/2 L  \mapright{\varphi}      L/ 2 N_\cE \subset   N_\cE/2 N_\cE \subset 
   \ns X/2\ns X \right).
     \end{align*} 
     Here primitivity is used to establish the rightmost inclusion. 
    
     The first consequence of this description  is  a bound for $\dim C_\cE$. Since $\ns X $ is primitive
 in $H^2(X,\bZ)$,  the quotient $\ns X/2\ns X$ injects into  $ H^2(X,\bZ)/2 H^2(X,\bZ)= H^2(X,\bbF_2)$, 
 a  symplectic inner product  space in which  the  image of $\varphi$ is totally isotropic and so  
 has dimension $\le \half  b_2(X)$.  It follows that 
 \begin{equation}
 \label{eqn:BoundCodes}
 \dim C_\cE \ge n- \half  b_2(X).
 \end{equation}  
Secondly, using the notion of  an "even set of nodal curves", which, I recall, 
    means that the   sum of the nodal curves is divisible by $2$ in     $\ns X$,  there is a further consequence:
    
  \begin{lemma}  \label{lem:weightsinNodCodes}   Non-zero code words in $C_\cE$ correspond to
 even subsets of disjoint nodal curves, and conversely.   More precisely, with $e_1,\dots,e_n$ the standard basis of $\bbF_2^n$, the sum $\sum_{i\in J}    e_i$
 belongs to the code $C_\cE$  if and only if   $\sum_{i\in J} E_i$ is   even in $\ns X$.
 The weight of a word in $C_\cE$ is the cardinality of the corresponding set.
    \end{lemma}

\section{Applying coding theory to nodal K3 surfaces}
\label{sec:k3}

   There are severe restrictions  on even sets of disjoint nodal curves on a complex K3 surface:     
 
 \begin{lemma} Let $X$ be a  complex  K3 surface containing an  even set of $k$ disjoint nodal curves. Then $k=0,8$ or $16$. 
 If $k=8$ the associated double cover is a K3 surface and if $k=16$ it is a complex torus. 
 \label{lem:EvenSetsOnK3}
 \end{lemma}
 
 \begin{proof} Let $\cE=\set{E_1,\dots,E_k}$ be an even set of double curves on $X$ and  from the  corresponding double cover. The inverse images of the 
 double curves are exceptional curves and blowing these down results in a minimal surface, say $Y_k$.  Since  the Euler number  of $X$ equals $e(X)=24$,
 one calculates easily that $e(Y_k )= 48 - 3k$. 
 
 On the other hand,    the canonical bundle of $Y_k$  is  trivial\footnote{ Indeed, the canonical bundle is trivialized on $X$ by a non-zero holomorphic two-form.
 On the double cover it lifts  as a holomorphic  two-form which is non-zero   outside  the branch locus and descends a holomorphic two-form $\omega$
  on $Y_k$, nowhere zero except maybe in   the points $p_j$. But  a section of a line bundle can at most have zeros  along a divisor and so $\omega$ trivializes  $K_{Y_k }$.}  so that  $p_g(Y_k )=1$
  and $c_1^2(Y_k )=0$. 
  The classification theorem~\cite[Ch. VI, Table 10]{4authors} then gives two possibilities:
  $ Y_k$  is either a torus or a K3-surface.      Combining this with Noether's  formula 
   \[
  \frac 1 {12} \cdot e(Y_k) = 2 - q(Y_k) =\frac 1 {12} (48 - 3k),
  \]
  gives two possibilities: either $k=8$ and then $Y_k$ is a K3 surface, or $k=16$ and then $Y_k$ is a torus.
 \end{proof}

\begin{rmk} The preceding argument is also valid for algebraic K3 surfaces in characteristics different from  $2$ since in that case there are only "classical" Enriques
 and bi-elliptic surfaces having the same invariants as in characteristic $0$. Compare the table on page 373 of \cite{bomhus}.\label{rmk:charp}
\end{rmk}

Combining Lemma~\ref{lem:weightsinNodCodes},  Lemma~\ref{lemm:beaucodes}  and    \eqref{eqn:BoundCodes}   one deduces:

  \begin{corr}  \label{cor:wgtsNik}
 For a set $\cE$  of disjoint nodal curves on a  complex K3 surface the  weights of the   associated code  are $0,8$ or $16$.
 If $\cE$ consists of $ 16$ nodal curves,  the associated code $C_{\cE}$  is  isomorphic to $\mathsf D_5$ where $\mathsf D_5$ is the Reed--Muller code~\eqref{eqn:CodeD}. 
If $\cE$ consists of $8$ nodal curves the code consists of the line in $\bbF^8_2$ spanned by $(1,\dots,1)$.
  \end{corr}
\begin{proof} The first assertion is a direct translation of  Lemma~\ref{lem:EvenSetsOnK3} using Lemma~\ref{lem:weightsinNodCodes}.

Assume that $\#\cE=16$.
The estimate \eqref{eqn:BoundCodes}  for the code $C_{\cE }\subset \bbF_2^{16}$ states   that $m=\dim C_{\cE } \ge 16 -  \half\cdot 22= 5$.
Then, since the  non-zero weights are $\ge 8$, Lemma~\ref{lemm:beaucodes} implies  $16\ge 2^{m-1}\ge 2^4$. Hence, since then  equality holds, $C_{\cE }=\mathsf D_5$.

The assertion for $k=8$ is clear.
\end{proof}

Since the code $\mathsf D_5$ contains the word  $(1,\ldots,1)$,    
  the sum of all the nodal curves  is even and so   any set of $16$ disjoint nodal curves on a K3 surface $X$  is an even set. 
  This reproves a  result by V.\ V.\ Nikulin~\cite{nik2}.
  Forming the double cover gives a complex two-torus
 blown up in $16$ points and the quotient by the standard involution is a Kummer surface whose minimal resolution of singularities is $X$.
Hence:
    
 \begin{prop} Let $X$ be a  complex   K3 surface  containing  a set   $\cE $  of  $16$ disjoint nodal curves.
 Then  
  $\cE $  is an even set,   and $X$  is the minimal resolution of  a Kummer surface.  Moreover, the primitive sublattice of $\ns X$ spanned 
  by $\cE$ is isometric to the lattice $\Gamma_{D_5}(-1)$.
 \label{prop:kummer}
 \end{prop}
 
 The lattice spanned by $16$ disjoint nodal curves on a (desingularised) Kummer surface  is also called the \emph{Kummer lattice}.
 Its characterization as  the abstract lattice $\Gamma_{\mathsf D_5}(-1)$ makes it possible to show the main result of this note:
 \begin{thm} A K3 surface cannot contain more than $16$ disjoint nodal curves.
 \end{thm}
 \begin{proof} Suppose there is a set $\cE$ consisting of  $n>16$ disjoint nodal curves on the surface. 
 A subset of $16$ nodal curves define a coordinate subspace of the code $C_\cE\subset \bbF_2^n$, which, by Corollary~\ref{cor:wgtsNik} must be  isomorphic to $\mathsf D_5$.
 Lemma~\ref{lemm:main} then implies the result.
 \end{proof}
 Suppose $\bar X$ is a compact complex surface with at most du Val singularities,  i.e.\ double points of type $A_n$-$D_n$-$E_n$. The corresponding lower case letter denoting
 the number of each type, put 
 \[
\delta(X) =\sum  (a_n+e_n)\cdot \left[ {n+1}\over 2\right]+d_n \cdot  
             \left[ n+2\over 2\right],
             \]
             where $X$ is the minimal resolution of singularities of $\bar X$.
             The number $\delta(X)$ gives the number of disjoint nodal curves on $X$ coming from    desingularizing $\bar X$  and so one finds: 
 \begin{corr} If $X$ is a K3 surface which is the  minimal resolution of singularities of  a surface  having at most du Val singularities, then $\delta( X)\le 16$.
 \end{corr}
 
 In particular, $\bar X$ can have at most $16$ ordinary nodes, or $A_1$-singularities. More generally,  $\delta(X)=16$ if there are only $A_1$ or $A_2$ singularities, but
 otherwise it is strictly smaller.   For instance, one can have at most four singularities which all  are  of types  $A_{16}, D_6$, $D_7,E_7$ or $E_8$. Also note the discrepancy with the Milnor number
 $\mu(X)=\sum n (a_n+d_n+e_n)$. Indeed  the more singularities with high Milnor number, the closer  $\delta(X)/\mu(X)  $ gets to $\half$.
 
\begin{rmk} \label{rmk:refining}     One can say more for specific types of projective K3 surfaces. For instance, a degree $6$   K3 surface in $\bP^5$, necessarily
             a complete intersection of a quadric and a cubic,    can have no more than $15$ double points.    See   \cite{kalker}.                                                                                                                                                                                                                                                                                                                                                                                                                                                                                                                                                                                                                                      
\end{rmk}

\end{document}